\documentclass[a4paper]{amsart}
\usepackage{mymaths}

\newcommand{\B}{\mathcal{B}}
\newcommand{\K}{\mathcal{K}}
\newcommand{\G}{\mathcal{G}}
\renewcommand{\H}{\mathcal{H}}
\DeclareMathOperator{\Prim}{Prim}
\DeclareMathOperator{\Ind}{Ind}
\DeclareMathOperator{\-Ind}{-Ind}

\newcommand{\Diff}{\mathrm{Diff}}

\renewcommand{\S}{\mathbb{S}}

\newcommand{\dr}{\rightrightarrows}

\renewcommand{\hat}{\widehat}

\title{The Fredholm property for groupoids is a local property}
\author{Rémi Côme}

\begin{document}

\maketitle

\begin{abstract}
  Fredholm Lie groupoids were introduced by Carvalho, Nistor and Qiao as a tool for the study of partial differential equations on open manifolds. This article extends the definition to the setting of locally compact groupoids and proves that \enquote{the Fredholm property is local}. Let $\G \dr X$ be a topological groupoid and $(U_i)_{i\in I}$ be an open cover of $X$. We show that $\G$ is a Fredholm groupoid if, and only if, its reductions $\G^{U_i}_{U_i}$ are Fredholm groupoids for all $i \in I$. We exploit this criterion to show that many groupoids encountered in practical applications are Fredholm. As an important intermediate result, we use an induction argument to show that the primitive spectrum of $C^*(\G)$ can be written as the union of the primitive spectra of all $C^*(\G|_{U_i})$, for $i \in I$.
\end{abstract}

\tableofcontents

\section{Introduction}
\label{sec:introduction}

\subsection{Motivations: differential equations on singular manifolds}
\label{sub:Motivations}
This paper deals with the study of locally compact groupoids, and more specifically with the structure of the primitive spectrum of their associated $C^*$-algebras. Nevertheless, our underlying motivation is the study of linear differential equations on open manifolds or on manifolds with singularities. 

Thus, let $M_0$ be a smooth and complete Riemannian manifold without boundary, and $P$ an order-$m$ differential operator on $M_0$. Let $(H^s(M_0))_{s \in \R}$ be the usual scale of Sobolev spaces on $M_0$ \cite{HR2008}. Under some natural conditions, the operator $P$ extends as a bounded operator
\begin{equation} \label{eq:P_btw_Sobolev_space}
  P : H^s(M_0) \to H^{s-m}(M_0).
\end{equation}

When $M_0$ is a \emph{closed} manifold, it is well-known that the operator in \eqref{eq:P_btw_Sobolev_space} is Fredholm if, and only if, it is \emph{elliptic}, meaning that its principal symbol $\sigma(P) \in C^\infty(T^*M_0)$ vanishes only on the zero section \cite{Hor1985}. This result has deep consequences concerning spectral theory, differential equations and index theory on closed manifolds \cite{Hor1985,AS1968,Gil1984}.

A natural and important question is to seek extensions of this Fredholm characterization for \emph{open} manifolds. In \cite{CNQ2017}, Carvalho, Nistor and Qiao introduced a very large class of manifolds, called \emph{manifolds with amenable ends}. To any manifold $M_0$ belonging to this class, we can associate a family of manifolds $M_\alpha$, for some $\alpha \in A$, which are acted upon by Lie groups $G_\alpha$, and such that the following theorem holds.

\begin{thm}[ {\cite[Theorem 1.1]{CNQ2017}} ] \label{thm:Fredholm_condition_for_operators_vague}
  Let $P$ be a \enquote{compatible} operator on $M_0$. Then one can associate some $G_\alpha$-invariant differential operators $P_\alpha$ on $M_\alpha$ such that $P : H^s(M_0) \to H^{s-m}(M_0)$ is Fredholm if, and only if,
  \begin{enumerate}
    \item $P$ is elliptic, and
    \item each operator $P_\alpha : H^s(M_\alpha) \to H^{s-m}(M_\alpha)$ is invertible.
  \end{enumerate}
\end{thm}
This very vague statement is made more precise in Theorem \ref{thm:fredholm_condition_for_operators} below. The operators $P_\alpha$ should be thought of as \enquote{limit operators} giving some control on the behaviour of $P$ at infinity. Note that Theorem \ref{thm:Fredholm_condition_for_operators_vague} remains true if we replace $P$ by an operator in a suitable pseudodifferential calculus or if we consider operators acting between vector bundles \cite{NWX1999,CNQ2017}.

Theorem \ref{thm:Fredholm_condition_for_operators_vague} recovers in a unified setting many similar results that were previously known in particular cases \cite{DLR2015,Sch1991,Dau1988,Kon1967,KMR1997, GI2006, MPR2007,Man2002}. This was made possible by noting that, in many cases, the relevant differential operators are generated by the action of a Lie groupoid $\G$ on the manifold $M_0$, in a sense made more precise below. Obtaining Fredholm conditions can then be reduced to studying the representations of the reduced $C^*$-algebra of $\G$ \cite{CNQ2017,NP2017}.

This motivated the definition of \emph{Fredholm groupoids} in \cite{CNQ2017}. Roughly speaking, a Fredholm groupoid is a Lie groupoid $\G$ whose unit space is a compact manifold with boundary $M$, and such that 
\begin{equation*}
  \text{$a \in 1 + C^*_r(\G)$ is Fredholm} \Leftrightarrow \text{$\pi_x(a)$ is invertible for any $x \in \d M$}.
\end{equation*}
Here $\pi_x$ stands for the regular representation of $\G$ at $x$. This is very similar in spirit to Theorem \ref{thm:Fredholm_condition_for_operators_vague}: an element is Fredholm if, and only if, some limit operators are invertible. 

Let us summarize this method. To a manifold $M_0$, seen as the interior of a compact manifold with boundary $M$, one associates a Lie groupoid $\G \dr M$ whose action generates an interesting subalgebra of differential operators on $M_0$. To obtain the Fredholm characterization of Theorem \ref{thm:Fredholm_condition_for_operators_vague}, the main objective is to prove that the groupoid $\G$ is Fredholm. The aim of this papers is to construct a large class of Fredholm groupoids.

\subsection{Main results}
\label{sub:Main results}
To anticipate some further applications of the theory, we start by generalizing the definition of Fredholm groupoids to the case of locally compact groupoids. Indeed, for many practical applications, it seems that it would be interesting to extend the framework of Theorem \ref{thm:Fredholm_condition_for_operators_vague} and the tools surrounding it to the setting of continuously family groupoids \cite{Pat2000,LMN2000}.

If $\G \dr X$ is a topological groupoid and $U \subset X$ an open subset, we denote by $\G|_U = \G_U^U$ the \emph{reduction} of $\G$ to $U$, that is the subgroupoid of $\G$ made of all elements whose domain and range lie in $U$. The \emph{saturation} of $U$ is the set $\G\cdot U = r(d^{-1}(U))$. Our point in this paper is that the Fredholm property for groupoids is a local property, in the following sense.

\begin{thm} \label{thm:main}
  Let $\G\dr X$ be a locally compact, second-countable and locally Hausdorff groupoid, endowed with a right-invariant Haar system. Assume that
  \begin{enumerate}
    \item there is an open dense $\G$-invariant subset $V \subset X$ with $\G_V \simeq V\times V$, and
    \item we have a family $(U_i)_{i\in I}$ of open subsets of $X$ such that the saturations $(\G\cdot U_i)_{i\in I}$ provide an open cover of $X$.
  \end{enumerate}
  Then $\G$ is a Fredholm groupoid if, and only if, each reduction $\G|_{U_i}$ is also a Fredholm groupoid, for $i \in I$.
\end{thm}

Most results giving sufficient conditions for a groupoid $\G$ to be Fredholm assume that $\G$ is Hausdorff, which sometimes is not so easy to check \cite{NP2017,CNQ2017}. This is not a requirement for Theorem \ref{thm:main}, which states that it is enough to look at the local structure of $\G$ (i.e.\ reductions). Since most groupoids studied in practical cases are locally very simple, this gives a powerful tool to prove the Fredholm property.

To study the representation theory of $\G$ in terms of its reductions, we will use the induction theory of $C^*$-algebras \cite{Rie1974,RW1998}. If $\G \dr X$ is a locally compact groupoid satisfying some extra assumptions, we associate to each open subset $U \subset X$ a continuous induction map between the primitive spectra
\begin{equation*}
  \Ind_U : \Prim(C^*(\G|_U)) \to \Prim(C^*(\G)),
\end{equation*}
which is an homeomorphism onto its image. As a important step, we obtain the following result
\begin{thm} \label{thm:second_main}
  Let $\G \dr X$ be a locally compact, second countable, and locally Hausdorff groupoid. Assume that we have a family of open subsets $(U_i)_{i\in I}$ such that their saturations $(\G\cdot U_i)_{i \in I}$ is an open cover of $X$. Then 
  \begin{equation*}
    \Prim C^*(\G) = \bigcup_{i \in I} \Ind_{U_i}\left( \Prim C^*(\G|_{U_i}) \right).
  \end{equation*}
\end{thm}
Theorem \ref{thm:second_main} gives us a good description of the representation theory of $\G$ in terms of that of its reductions. This is the main tool for the proof of Theorem \ref{thm:main}.

\subsection{Applications and examples}
\label{sub:Applications and examples}

A direct consequence of Theorem \ref{thm:main} is that gluing Fredholm groupoids together yields another Fredholm groupoid, which generalizes a result of \cite{Com2018}. Moreover, most groupoids appearing in practical applications have a very simple local structure: locally they are reductions of action groupoids. To formalize this, we introduce the class of \emph{local action groupoids}, and prove some related results. In particular, we show that if a local action groupoid $\G$ is locally given by the action of an \emph{amenable} group, then the groupoid $\G$ is Fredholm. 

Concrete example include the groupoids associated to manifolds with cylindrical ends, asymptotically euclidean manifolds and asymptotically hyperbolic manifolds. We also introduce a groupoid which models the analysis on manifolds with cuspidal points and prove that it is Fredholm.

\subsection{Outline of the paper}
\label{sub:Outline of the paper}

We begin in Section \ref{sec:preliminaries} by recalling some known results concerning the primitive spectrum of a $C^*$-algebra and the induction mechanism. We then turn our attention to locally compact groupoids and introduce some definitions and notations.

Section \ref{sec:representations_induced_from_reductions} introduce our main tool, which is the induction functor from the $C^*$-algebra of a reduction to an open subset. We define this functor and establish some important properties, and then use it to prove the decomposition of the primitive spectrum stated in Theorem \ref{thm:second_main}.

It is in Section \ref{sec:application_to_fredholm_groupoids} that we start dealing with Fredholm groupoids. We introduce our definition of Fredholm groupoids in the locally compact case and prove Theorem \ref{thm:main}. We then establish a few consequences. Among them, we define the notion of \emph{local isomorphisms} of two groupoids and the class of \emph{local action groupoids}.

Finally, Section \ref{sec:Applications to differential equations} gives some concrete examples of Fredholm groupoids. To motivate their construction, we step back into the setting of Lie groupoids and recall the link with the study of differential operators on manifolds. We then show how Theorem \ref{thm:main} may be used to prove that the groupoids under study are Fredholm (and local action groupoids).

\vspace{1em}

\emph{Aknowledgements:} the author would like to thank Victor Nistor for useful discussions and suggestions.

\section{Preliminaries}
\label{sec:preliminaries}

\subsection{The primitive spectrum of a $C^*$-algebra}
\label{sub:primitive_spectrum}

We recall in this section the definition of the primitive spectrum of a $C^*$-algebra, as well as the general induction mechanism for representations. The reader interested in more details may refer to the book of Dixmier \cite{Dix1977} for more details on the primitive spectrum and to \cite{Rie1974,RW1998} for the induction procedure.

\begin{defn}
  Let $A$ be a $C^*$-algebra. An ideal $J \subset A$ is called \emph{primitive} if it is the kernel of a non-zero irreducible representation of $A$. The \emph{primitive spectrum} of $A$, denoted $\Prim A$, is the set of all primitive ideals in $A$.
\end{defn}

For any ideal $I \subset A$, let 
\begin{equation*}
  \Prim_I A = \{ J \in \Prim A, I \subset J\},
\end{equation*}
and denote by $\Prim^I A$ the complement of $\Prim_I A$ in $\Prim A$. The sets $\Prim^I A$, where $I$ ranges through the ideals of $A$, are precisely the open sets in the \emph{Jacobson topology} of $\Prim A$. The \emph{support} of a representation $\pi$ of $A$ is the closed subset
\begin{equation*}
  \supp \pi \eqdef \Prim_{\ker \pi} A = \{J \in \Prim A, \ker \pi \subset J\}.
\end{equation*}

For any $C^*$-algebra $A$, let $\mathcal{R}(A)$ denote the category of unitary equivalence classes of non-degenerate representations of $A$. A well-known result states that, if $I$ is an ideal of $A$ and $\pi$ a non-degenerate representation of $I$ on a Hilbert space $H$, then there is a unique representation of $A$ on $H$ extending $\pi$ \cite{Dix1977}. This defines an induction functor
\begin{equation*}
  \Ind_I^A : \mathcal{R}(I) \to \mathcal{R}(A).
\end{equation*}
Moreover, the representation $\Ind_I^A \pi$ is irreducible if, and only if $\pi$ is irreducible. Thus $\Ind_I^A$ descends to a map
\begin{equation*}
  \Ind_I^A : \Prim I \to \Prim A,
\end{equation*}
which is a homeomorphism onto $\Prim^I A$ \cite{Dix1977}.

The ideal $I$ is a particular case of an \emph{$(A,I)$-correspondence} in the sense of \cite{Rie1974,Ren2014}. If $A,B$ are $C^*$-algebras, a $(A,B)$-correspondence $\mathcal{E}$ is a full right Hilbert $B$-module endowed with a morphism $\pi : A \to \mathcal{L}_B(\mathcal{E})$ such that $\pi(A)\mathcal{E}$ is dense in $\mathcal{E}$ (it corresponds to the notion of a $B$-rigged $A$-module in \cite{Rie1974}). To any such correspondence is associated an induction functor
\begin{equation*}
  \mathcal{E}\-Ind : \mathcal{R}(B) \to \mathcal{R}(A)
\end{equation*}
given by the tensor product with $\mathcal{E}$. If $\mathcal{E}$ is moreover an $(A,B)$-\emph{imprimitivity bimodule}, i.e.\ both a full left $A$-module and a full right $B$-module satisfying some compatibility conditions \cite{Rie1974,RW1998}, then $A$ and $B$ are said to be Morita equivalent and $\mathcal{E}\-Ind$ is an equivalence of categories. In that case, the functor $\mathcal{E}\-Ind$ descends to an homeomorphism
\begin{equation*}
  \mathcal{E}\-Ind : \Prim B \to \Prim A.
\end{equation*}

\subsection{Groupoids and their $C^*$-algebras}
\label{sub:groupoids}

We now turn our attention to topological groupoids and fix some definitions and notations. Recall that a groupoid consits of two sets: the set of \emph{arrows} $\G$ and the set of \emph{units} $X$, together with five structural morphisms: the domain and range $d, r : \G \to X$, the inverse $\iota: \G \to \G$, the inclusion of units $u : X \to \G$ and the product $\mu$ from the space 
\begin{equation*}
  \G^{(2)} \eqdef \{ (g,h) \in \G\times\G, d(g) = r(h) \}
\end{equation*}
of \emph{compatible arrows} to $\G$. Throughout the paper, we denote by $\G \dr X$ a groupoid with set of units $X$. If $A$ is a subset of $X$, we denote by $\G^A = r^{-1}(A)$ and $\G_A = d^{-1}(A)$. The groupoid $\G|_A \eqdef \G^A \cap \G_A$, with units $A$, is the \emph{reduction} of $\G$ to $A$. Finally, we denote by
\begin{equation*}
  \G\cdot A \eqdef \{ r(g) \mid g\in \G, d(g) \in A \} = r(d^{-1}(A))
\end{equation*}
the \emph{saturation} of $A$ in $X$ through the action of $\G$. The reader wishing to learn more about groupoids should refer to \cite{Ren1980,Mac1987}.

\begin{defn} \label{defn:topological_groupoid}
  A \emph{topological groupoid} is a groupoid $\G \dr X$ such that
  \begin{enumerate}
    \item the sets $\G$ and $X$ are topological spaces, with $X$ Hausdorff,
    \item all five structural maps $d,r,\iota,u$ and $\mu$ are continuous,
    \item the domain map $d$ is open.
  \end{enumerate}
\end{defn}
It follows from Definition \ref{defn:topological_groupoid} that $\iota$ is a homeomorphism and $r$ is open as well. We shall usually require the topological space $\G$ to be locally compact, second countable and locally Hausdorff (in the sense that each element of $\G$ should have a Hausdorff neighborhood).

If $\G$ is Hausdorff, let $C_c(\G)$ be the space of $\C$-valued continuous functions with compact support in $\G$. If $\G$ is only locally Hausdorff, then $C_c(\G)$ denotes instead the linear space generated by continuous functions that are compactly supported in a Hausdorff subset of $\G$ (see \cite{Con1982} for why this is a better choice). 

To define a product on $C_c(\G)$, we recall below the standard notion of a \emph{Haar system} from \cite{Ren1980}. In the following definition, we denote by $R_g : \G_{r(g)} \to \G_{d(g)}$ the homeomorphism induced by right-multiplication by an element $g \in \G$.

\begin{defn}
  Let $\G \dr X$ be a locally compact groupoid. A continuous, right-invariant \emph{Haar system} on $\G$ is a family of Borel measures $(\lambda_x)_{x \in X}$ such that
  \begin{enumerate}
    \item for all $x \in X$, the support of $\lambda_x$ is $\G_x = d^{-1}(x)$,
    \item (right-invariance) for any $g \in \G$, we have $(R_g)_*\lambda_{r(g)} = \lambda_{d(g)}$,
    \item (continuity) for any $f \in C_c(\G)$, the function
      \begin{equation*}
        x \mapsto \int_{\G_x} f(g) d\lambda_x(g)
      \end{equation*}
      is continuous.
  \end{enumerate}
\end{defn}

Assume that $\G$ is endowed with a continuous, right-invariant Haar system. If $f,g \in C_c(\G)$, we define their convolution product $f*g \in C_c(\G)$ by
\begin{equation*}
  f*g(x) = \int_{\G_{d(x)}} f(xy^{-1}) g(y) d\lambda_{d(x)}(y).
\end{equation*}

The \emph{full} $C^*$-algebra of $\G$, denoted $C^*(\G)$, is the completion of $C_c(\G)$ for the norm
\begin{equation*}
  \|f\|_{C^*(\G)} = \sup_\pi \|\pi(f)\|,
\end{equation*}
where $\pi$ ranges over all continuous, bounded representations of $C_c(\G)$ \cite{Ren1980}. The \emph{reduced} $C^*$-algebra of $\G$, denoted $C^*_r(\G)$, is the completion of $C_c(\G)$ for the norm
\begin{equation*}
  \|f\|_{C^*_r(\G)} = \sup_{x \in X} \|\pi_x(f)\|.
\end{equation*}
Here $\pi_x : C_c(\G) \to \B(L^2(\G_x))$ is the \emph{regular representation} of $\G$ at $x$, defined by $\pi_x(f)\xi = f*\xi$.

Finally, let $\G \dr X$ and $\H \dr Y$ be two locally compact, second-countable and locally Hausdorff groupoids. Recall that $\G$ and $\H$ are \emph{Morita equivalent} if there exists a locally compact space $Z$ with a left $\G$-action and a right $\H$-action, such that both actions are free and proper, commute with each other and the anchors $Z \to X$ and $Z \to Y$ induces bijections $\G\backslash Z \simeq Y$ and $Z/\H \simeq X$. Morita equivalent groupoids have Morita equivalent $C^*$-algebras \cite{MRW1987, Tu2004}. It follows from Subsection \ref{sub:primitive_spectrum} that there are homeomorphisms $\Prim C^*(\G) \simeq \Prim C^*(\H)$ and $\Prim C^*_r(\G) \simeq \Prim C^*_r(\H)$.

\section{Primitive spectrum and groupoid reductions}
\label{sec:representations_induced_from_reductions}

In this section, we show that the primitive spectrum of a groupoid $C^*$-algebra can be investigated locally: this is the content of Theorem \ref{thm:decomposition_of_the_spectrum}, which will be our main tool for Section \ref{sec:application_to_fredholm_groupoids}. Throughout the section, we shall consider a locally compact, second-countable and locally Hausdorff groupoid $\G \dr X$ that is endowed with a right-invariant continuous Haar system. 

\subsection{Representations induced from a reduction}
\label{sub:induced_representations}

We will show that each reduction of the groupoid to an open subset $U\subset X$ defines an induction functor between the categories of representations of $\G|_U = d^{-1}(U)\cap r^{-1}(U)$ and $\G$. The starting point for our construction is Remark \ref{rem:reduction_eq_saturation} below.

\begin{rem} \label{rem:reduction_eq_saturation}
  If $U \subset X$ is an open subset and $W \eqdef \G\cdot U$ its saturation, then the reduction $\G|_U$ is Morita equivalent to $\G_W$. The equivalence is implemented by the $(\G_W,\G|_U)$-space $\G_U = d^{-1}(U)$, acted upon by left and right multiplication. Both actions are free and proper, and the domain and range maps induce isomorphisms $\G_W \backslash \G_U \simeq U$ and $\G_U/\G|_U \simeq W$.
\end{rem}

According to the results recalled in Section \ref{sec:preliminaries}, it follows that $C^*(\G|_U)$ and $C^*(\G_W)$ are Morita equivalent \cite{MRW1987}. The corresponding $(C^*(\G_W),C^*(\G|_U))$-imprimitivity bimodule $\mathcal{E}_U$ is the completion of $C_c(\G_U)$ for the norm
\begin{equation*}
  \|f\|_{\mathcal{E}_U} = \|f^* * f\|_{C^*(\G|_U)}^\frac{1}{2},
\end{equation*}
with $C^*(\G_W)$ and $C^*(\G|_U)$ acting by right and left multiplication respectively. Similarly, there is a Morita equivalence between the reduced algebras $C^*_r(\G_U)$ and $C^*_r(\G_W)$. We choose to stick to the study of the full algebras for now, but all results of this section apply to their reduced counterparts, as pointed by Remark \ref{rem:extends_to_reduced_algebras}.

It is well-known that for any $\G$-invariant open subset $W \subset X$, the $C^*$-algebra $C^*(\G_W)$ embeds as an ideal in $C^*(\G)$ \cite{MRW1996}. We saw in Subsection \ref{sub:primitive_spectrum} that this implies the existence of an induction functor
\begin{equation*}
  \Ind_W : \mathcal{R}(C^*(\G_W)) \to \mathcal{R}(C^*(\G)),
\end{equation*}
between the respective categories of non-degenerate representations.

\begin{defn} \label{defn:induction}
  Let $\G \dr X$ be a locally compact, second countable and locally Hausdorff groupoid. Let $U$ be an open subset of $X$, and $W \eqdef \G\cdot U$ be its saturation. The \emph{induced representation} functor
\begin{equation*}
  \Ind_U : \mathcal{R}(C^*(\G|_U)) \to \mathcal{R}(C^*(\G))
\end{equation*}
is defined as the composition $\Ind_U = \Ind_W\circ \mathcal{E}_U\-Ind$.
\end{defn}

\begin{rem}
  A possibly more direct way to define the functor $\Ind_U$ is to observe that $\mathcal{E}_U$ is a $(C^*(\G),C^*(\G|_U))$-correspondence, as introduced in Subsection \ref{sub:primitive_spectrum}. Correspondingly, the space $\G_U$ is a $(\G,\G|_U)$-correspondence (or Hilsum-Skandalis morphism) in the sense of \cite{Ren2014, HS1987}.
  We do not emphasize this approach too much however, because the factorization of $\Ind_U$ through $\mathcal{R}(C^*(\G_{W}))$ given by Definition \ref{defn:induction} will be of use below.
\end{rem}

\begin{lm} \label{lm:induction_norm_estimate}
  Let $\G \dr X$ be a locally compact, second-countable and locally compact groupoid, and let $U \subset X$ be open. If $\pi$ is a non-degenerate representation of $C^*(\G|_U)$, then for any $f \in C_c(\G|_U)$ we have
  \begin{equation*}
    \|\pi(f)\| \le \| \Ind_U \pi(f)\|.
  \end{equation*}
\end{lm}

\begin{proof}
  Recall that $\Ind_U \pi$ is a representation of $C^*(\G)$ on $\mathcal{H} \eqdef \mathcal{E}_U \otimes_\pi H$. Let $\mathcal{H}_U \subset \mathcal{H}$ be the closed subspace generated by linear combinations of tensors $f\otimes \xi$, with $\xi \in H$ and $f \in C_c(\G|_U)$. Let also $\rho$ be the restriction of $\Ind_U \pi$ to $C_c(\G|_U)$. If $f,g \in C_c(\G|_U)$ and $\xi, \eta \in H$, then $\rho(f)(g\otimes\xi) = (f*g)\otimes \xi$ is in $\mathcal{H}_U$, so $\mathcal{H}_U$ is stable by $\rho$. Moreover
  \begin{equation*}
    \langle f\otimes\xi, g\otimes\eta \rangle_{\mathcal{H}_U} \eqdef \langle \pi(g^* * f)\xi, \eta \rangle_H = \langle \pi(f)\xi, \pi(g)\eta \rangle_H,
  \end{equation*}
  so the map $\Phi : f\otimes \xi \mapsto \pi(f)\xi$ extends as an isometry from $\mathcal{H}_U$ to $H$. Since $\pi$ is non-degenerate, the map $\Phi$ is an isomorphism. Now $\Phi(\rho(f)g\otimes\xi) = \pi(f)\Phi(f\otimes\xi)$, so $\rho$ and $\pi$ are unitarily equivalent. This shows that 
  \begin{equation*}
    \|\pi(f)\|  = \|\rho(f) \| \le \|\Ind_U \pi(f)\|
  \end{equation*}
  for any $f \in C_c(\G|_U)$.
\end{proof}

A corollary is the following not-so-obvious result, for which we could not find a reference in the litterature.
\begin{cor}
  Let $\G \dr X$ be a locally compact, second countable and locally Hausdorff groupoid. Let $U$ be an open subset of $X$. Then $C^*(\G|_U)$ is a subalgebra of $C^*(\G)$.
\end{cor}
\begin{proof}
  Let $f \in C_c(\G|_U)$. Every representation of $C^*(\G)$ restricts as a bounded representation of $C_c(\G|_U)$, so $\|f\|_{C^*(\G)} \le \|f\|_{C^*(\G|_U)}$. Conversely, if $\pi$ is a non-degenerate representation of $C^*(\G|_U)$, then $\Ind_U\pi$ is a representation of $C^*(\G)$ such that
  \begin{equation*}
    \|\pi(f)\| \le \|\Ind_U(f)\|
  \end{equation*}
  by Lemma \ref{lm:induction_norm_estimate}. Therefore $\|f\|_{C^*(\G|_U)} \le \|f\|_{C^*(\G)}$. This shows that the inclusion of $C_c(\G|_U)$ into $C_c(\G)$ extends to a continuous and injective $*$-morphism from $C^*(\G|_U)$ into $C^*(\G)$.
\end{proof}

For any open subset $U \subset X$, set
\begin{equation*}
  \Prim_U C^*(\G) \eqdef \{ J \in \Prim C^*(\G), C^*(\G_U) \subset J \},
\end{equation*}
and let $\Prim^U C^*(\G)$ be its complementary subset in $\Prim C^*(\G)$.

\begin{lm} \label{lm:prim_U_equals_prim_W}
  Let $U \subset X$ be open, and $W = \G\cdot U$ be its saturation. Then
  \begin{equation*}
    \Prim_U C^*(\G) = \Prim_W C^*(\G) \quad \text{and} \quad \Prim^U C^*(\G) = \Prim^W C^*(\G).
  \end{equation*}
\end{lm}
\begin{proof}
  The algebra $C^*(\G_W)$ is the closed, two-sided ideal generated by $C^*(\G|_U)$ in $C^*(\G)$. Thus, if $J$ is a primitive ideal of $C^*(\G)$ that contains $C^*(\G|_U)$, then $J$ also contains all of $C^*(\G_W)$. On the other hand, it is obvious that if $J$ contains $C^*(\G_W)$, then it also contains the subalgebra $C^*(\G|_U)$. This proves that $\Prim_U C^*(\G) =\Prim_W C^*(\G)$, and therefore that $\Prim^U C^*(\G) = \Prim^W C^*(\G)$.
\end{proof}

We can now record the main properties of $\Ind_U$.

\begin{prop} \label{prop:induction_main_properties}
  Let $\G \dr X$ be a locally compact, second countable and locally Hausdorff groupoid. Let $U$ be an open subset of $X$.
  \begin{enumerate}
    \item \label{it:prim} The functor $\Ind_U$ descends to a continuous map
      \begin{equation*}
        \Ind_U : \Prim C^*(\G|_U) \to \Prim C^*(\G),
      \end{equation*}
      which is an homeomorphism onto $\Prim^U C^*(\G)$.

    \item \label{it:weak_containment} Let $\pi, \rho$ be two non-degenerate representations of $C^*(\G|_U)$ such that $\pi$ is weakly contained in $\rho$. Then $\Ind_U \pi$ is weakly contained in $\Ind_U \rho$. 

    \item \label{it:support} If $\pi$ is a non-degenerate representation of $C^*(\G|_U)$, then 
      \begin{equation*}
        \Ind_U(\supp \pi) \subset \supp(\Ind_U \pi).
      \end{equation*}
    \item \label{it:regular_rep} For $x \in U$, let $\pi_x^U$ be the corresponding regular representation of $C^*(\G|_U)$ and $\pi_x$ the one of $C^*(\G)$. Then $\Ind_U \pi_x^U = \pi_x$.
  \end{enumerate}
\end{prop}
\begin{proof}
  According to Definition \ref{defn:induction}, the functor $\Ind_U$ is defined as the composition $\Ind_W\circ \mathcal{E}_U\-Ind$. We highlighted in Section \ref{sec:preliminaries} that both $\Ind_W$ and $\mathcal{E}_U\-Ind$ induce continuous maps between primitive spectra, that are homeomorphims onto their respective images. Therefore, the map
  \begin{equation} \label{eq:defn_induction}
    \begin{aligned}
      \Ind_U : \Prim C^*(\G|_U) & \to \Prim C^*(\G) \\
      \ker \pi & \mapsto \ker(\Ind_U \pi)
    \end{aligned}
  \end{equation}
  is well-defined, continuous, and an homeomorphism onto $\Prim^W C^*(\G)$. The latter coincides with $\Prim^U C^*(\G)$ by Lemma \ref{lm:prim_U_equals_prim_W}, which proves Assertion \eqref{it:prim}.

  Assertion \eqref{it:weak_containment} is a well-known property of the Rieffel induction procedure, whose proof can be found in \cite{RW1998}. Assertion \eqref{it:support} is a direct consequence. Indeed, let $J = \ker \rho$ be a primitive ideal contained in $\supp \pi$. By definition, this is equivalent to $\rho \prec \pi$. Then $\Ind_U \rho \prec \Ind_U \pi$, which means that $\ker(\Ind_U \pi) \subset \ker(\Ind_U \rho)$. Since $\Ind_U J = \ker(\Ind_U \rho)$ by Equation \eqref{eq:defn_induction}, we conclude that $\Ind_U J \in \supp(\Ind_U(\pi))$. This proves the inclusion $\Ind_U(\supp \pi) \subset \supp(\Ind_U \pi)$.

  To prove Assertion \ref{it:regular_rep}, let $\mathcal{H} = \mathcal{E}_U \otimes_{\pi_x^U} L^2(\G_x^U)$. We need to show that the map $\Phi : \mathcal{H} \to L^2(\G_x)$ defined by
  \begin{equation*}
    \Phi: f \otimes \xi \mapsto f * \xi
  \end{equation*}
  extends to a Hilbert space isomorphism. Let $f,g \in C_c(\G_U)$ and $\xi,\eta \in C_c(\G_x^U)$. By definition
  \begin{align*}
    \langle f\otimes \xi, g\otimes \eta \rangle_{\mathcal{H}} & = \langle (g^**f)*\xi,\eta \rangle_{L^2(\G^U_x)} = \langle (g^**f)*\xi,\eta \rangle_{L^2(\G_x)} \\ 
    & = \langle f*\xi,g*\eta \rangle_{L^2(\G_x)},
  \end{align*}
  so $\Phi$ is an isometry. To show that $\Phi$ is onto, let $f \in C_c(\G_U)$, and let $(\xi_n)_{n \in \N}$ be an approximate unit in $C_c(\G|_U)$. Then $(f*\xi_n)|_{\G_x}$ converges to $f|_{\G_x}$ in $L^2(\G_x)$. This proves that the image of $\Phi$ contains the dense subset $C_c(\G_x)$, hence $\Phi$ is onto. The map $\Phi$ is therefore an isomorphism. Now, let $\rho = \Ind_U \pi_x^U$. If $g \in C_c(\G)$, then 
  \begin{align*}
    \Phi(\rho(g)(f\otimes\xi)) & \eqdef \Phi((g*f)\otimes\xi) = g*(f*\xi) \\
     &= \pi_x(g)(f*\xi) = \pi_x(g) \Phi(f\otimes\xi).
  \end{align*}
  Since $C_c(\G)$ is dense in $C^*(\G)$ and $\rho$ and $\pi_x$ are continuous, this proves that $\Ind_U \pi_x^U$ and $\pi_x$ define the same class in $\mathcal{R}(C^*(\G))$.
\end{proof}

\begin{rem} \label{rem:extends_to_reduced_algebras}
  All the constructions of this section can be made in the same way by replacing every full groupoid algebras by their \emph{reduced} counterparts. More explicitely, if $\G \dr X$ is a groupoid satisfying the assumptions of Definition \ref{defn:induction} and $U \subset X$ an open subset, then there is an induction functor
  \begin{equation*}
    \Ind_U : \mathcal{R}(C^*_r(\G|_U)) \to \mathcal{R}(C^*(\G)).
  \end{equation*}
  All the properties from Proposition \ref{prop:induction_main_properties} follow if we replace each full agebra by its reduced counterpart.
\end{rem}

\begin{rem}
  It should be noted that $C^*(\G|_U)$ is a \emph{hereditary subalgebra} \cite{Bla2006} of $C^*(\G)$ (and so is $C^*_r(\G|_U)$ in $C^*_r(\G)$). A hereditary subalgebra $B \subset A$ is always Morita equivalent to the closed, two-sided ideal $I_B$ it generates: in our case this ideal is $I_{C^*(\G|_U)} = C^*(\G_W)$, with $W = \G\cdot U$. This yields an induction functor
  \begin{equation*}
    \Ind_B^A : \mathcal{R}(B) \to \mathcal{R}(A),
  \end{equation*}
  which factorizes through $\mathcal{R}(I_B)$, just as in Definition \ref{defn:induction}. This recasts our construction in a more general setting, for which most results we have shown in this subsection should hold.
\end{rem}

\subsection{Decomposition of the spectrum}
\label{sub:decomposition_of_the_spectrum}

As in the previous sections, let $\G \dr X$ be a locally compact, second countable and locally Hausdorff groupoid, endowed with a right-invariant continuous Haar system. If $f \in C_c(\G)$ and $\varphi \in C_0(X)$, we follow the notation of \cite{Ren1980} and denote by $\varphi f$ the function $(\varphi\circ r)\cdot f$ (the central dots denotes scalar multiplication, and not convolution).

\begin{lm} \label{lm:spectrum_decomposition_by_ideals}
  Let $A$ be a $C^*$-algebra and $(I_\lambda)_{\lambda \in \Lambda}$ a family of ideals in $A$ such that $\sum_{\lambda \in \Lambda} I_\lambda = A$. Then 
  \begin{equation*}
    \Prim A = \bigcup_{\lambda \in \Lambda} \Prim I_\lambda,
  \end{equation*}
  where we identify $\Prim I_\lambda$ with its image $\Prim^{I_\lambda} A$ through $\Ind_{I_\lambda}^A$.
\end{lm}
The reader should refer to Subsection \ref{sub:primitive_spectrum} for the definition of $\Prim^{I_\lambda} A$ and the induction map $\Ind_{I_\lambda}^A$.
\begin{proof}
  For all $J \in \Prim(A)$, there is a $\lambda \in \Lambda$ such that $I_\lambda \not\subset J$. Indeed, if that were not the case, then we would have $A = \sum_{\lambda\in\Lambda} I_\lambda \subset J$ so $J = A$, which is not a primitive ideal. Therefore there is a $\lambda \in \Lambda$ such that $J \in \Prim^{I_\lambda}(A)$, which proves the proposition.
\end{proof}

\begin{lm} \label{lm:multiplication_continuity}
  Let $\varphi \in C_0(X)$, and define $M_\varphi: C_c(\G) \to C_c(\G)$ by $M_\varphi(f) = \varphi f$. Then $M_\varphi$ extends as a continuous linear map from $C^*(\G)$ to itself. Moreover, if $U \subset X$ is a $\G$-invariant open subset of $X$ such that $\supp \varphi \subset U$, then $f \mapsto \varphi f$ extends as a continuous map from $C^*(\G)$ to $C^*(\G_U)$.
\end{lm}

The first part of the Lemma implies that $C_0(X)$ embeds into the multiplier algebra of $C^*(\G)$. 

\begin{proof}
  The first statement was proven in \cite[Proposition 1.14]{Ren1980}, in which it was shown that
  \begin{equation*}
     \|\varphi f\|_{C^*(\G)} \le \|\varphi\|_\infty \|f\|_{C^*(\G)}.
  \end{equation*}
  
  If $U \subset X$ is an open subset such that $\varphi \in C_c(U)$, then $\varphi\circ r \in C_c(\G^U)$. If $U$ is moreover $\G$-invariant, then $\G^U = \G_U$, so $\varphi f \in C_c(\G_U)$. We know from \cite{Ren1980} that $C^*(\G_U)$ is an ideal in $C^*(\G)$, so 
  \begin{equation*}
    \|\varphi f\|_{C^*(\G_U)} = \|\varphi f\|_{C^*(\G)} \le \|\varphi\|_\infty \|f\|_{C^*(\G)}.
  \end{equation*}
  This proves the continuity as a map to $C^*(\G_U)$.
\end{proof}

We are ready to prove one of the main theorems of this paper. Again, recall that the definition of $\Prim^U C^*(\G)$ and $\Ind_U$ were introduced in the previous subsection. 

\begin{thm} \label{thm:decomposition_of_the_spectrum}
  Let $\G \dr X$ be a locally compact, second countable and locally Hausdorff groupoid. Assume that we have a family of open subsets $(U_i)_{i\in I}$ such that their saturations $(\G\cdot U_i)_{i \in I}$ form an open cover of $X$. Then 
  \begin{equation*}
    \Prim C^*(\G) = \bigcup_{i \in I} \Prim C^*(\G|_{U_i}),
  \end{equation*}
  where we identify $\Prim C^*(\G|_{U_i})$ with its image $\Prim^{U_i} C^*(\G)$ through $\Ind_{U_i}$.
\end{thm}
Theorem \ref{thm:decomposition_of_the_spectrum} is a localization result: the primitive spectrum of $C^*(\G)$ can be fully described by restricting our attention to sufficiently many reductions of $\G$ to open subsets of the unit space.

\begin{proof}
  Put $W_i \eqdef \G\cdot U_i$, for each $i \in I$. The assumption is that $(W_i)_{i\in I}$ is an open cover of $X$, so let $(\varphi_i)_{i\in I}$ be a partition of unity subordinate to that cover. If $a \in C^*(\G)$, then it follows from Lemma \ref{lm:multiplication_continuity} that $\varphi_i a$ is well defined for all $i$ and belongs to $C^*(\G_{W_i})$. Since $\sum_{i \in I} \varphi_i = 1$, we have $a =\sum_{i \in I} \varphi_i a$. Thus
  \begin{equation*}
    C^*(\G) = \sum_{i\in I} C^*(\G_{W_i}),
  \end{equation*}
  which is all we need to apply Lemma \ref{lm:spectrum_decomposition_by_ideals}. Therefore
  \begin{equation*}
    \Prim C^*(\G) = \bigcup_{i \in I} \Prim^{W_i} C^*(\G),
  \end{equation*}
  and we conclude with the identification $\Prim^{W_i}C^*(\G) = \Prim^{U_i}C^*(\G)$ established in Proposition \ref{prop:induction_main_properties}.
\end{proof}

\begin{rem}
  As was already highlighted in Remark \ref{rem:extends_to_reduced_algebras}, all results from this paper remain valid if we replace the full groupoid algebras with their \emph{reduced} couterparts. More explicitely, under the assumptions of Theorem \ref{thm:decomposition_of_the_spectrum}, there is a decomposition
  \begin{equation*}
    \Prim C^*_r(\G) = \bigcup_{i \in I} \Prim C^*_r(\G|_{U_i}),
  \end{equation*}
  where we identify $\Prim C^*_r(\G|_{U_i})$ with its image $\Prim^{U_i} C^*_r(\G)$ through $\Ind_{U_i}$. Note that the technical Lemma \ref{lm:multiplication_continuity} is much more easy to prove in the reduced case. Indeed there is no need to use Renault's disintegration theorem here, since we only have to deal with the regular representations of $\G$.
\end{rem}

\begin{rem}
  It should also be noted that a decomposition similar to that of Theorem \ref{thm:decomposition_of_the_spectrum} holds for the full spectrum of $C^*(\G)$ (i.e.\ equivalence classes of irreducible representations as defined in \cite{Dix1977}). Under the assumptions of Theorem \ref{thm:decomposition_of_the_spectrum}, we may write
  \begin{equation*}
    \hat{C^*(\G)} = \bigcup_{i \in I} \hat{C^*(\G|_{U_i})}.
  \end{equation*}
  where $\hat{C^*(\G|_{U_i})}$ is identified with its image through $\Ind_{U_i}$.
\end{rem}

\subsection{Families of representations}
\label{sub:families_of_representations}

The main motivation for Theorem \ref{thm:decomposition_of_the_spectrum} is to study the representations of $C^*(\G)$ from the representations of its reductions. In particular, it was proven in \cite[Proposition 2.1]{NP2017} that a family of representations $\mathcal{F}$ of a $C^*$-algebra $A$ is faithful if, and only if
\begin{equation*}
  \Prim A = \bar{\bigcup_{\pi \in \mathcal{F}} \supp \pi}.
\end{equation*}

\begin{cor} \label{cor:faithful_induction}
  We follow the assumptions of Theorem \ref{thm:decomposition_of_the_spectrum}. For each $i \in I$, let $\mathcal{F}_i$ be a faithful family of non-degenerate representations of $C^*(\G|_{U_i})$. Then the family
  \begin{equation*}
    \mathcal{F} \eqdef \{ \Ind_{U_i} \pi \mid i \in I, \pi \in \mathcal{F}_i \}
  \end{equation*}
  is faithful for $C^*(\G)$.
\end{cor}
\begin{proof}
  By assumption, for all $i \in I$ we have
  \begin{equation*}
    \Prim C^*(\G|_{U_i}) = \bar{\bigcup_{\pi \in \mathcal{F}_i} \supp \pi_i}.
  \end{equation*}
  Using Theorem \ref{thm:decomposition_of_the_spectrum}, we get
  \begin{equation} \label{eq:prim_equals_union}
    \Prim C^*(\G) = \bigcup_{i\in I} \Prim^{U_i} C^*(\G) = \bigcup_{i\in I} \bar{\bigcup_{\pi \in \mathcal{F}_i} \Ind_{U_i}(\supp \pi)}.
  \end{equation}
  It was proven in Proposition \ref{prop:induction_main_properties} that $\Ind_{U_i}(\supp \pi) \subset \supp(\Ind_{U_i} \pi)$ for any non-degenerate representation $\pi$ of $C^*(\G|_{U_i})$. Thus
  \begin{equation*}
    \bigcup_{\pi \in \mathcal{F}_i} \Ind_{U_i}(\supp \pi) \subset \bigcup_{\pi \in \mathcal{F}_i} \supp(\Ind_{U_i} \pi)
  \end{equation*}
  Together with Equation \eqref{eq:prim_equals_union}, we obtain
  \begin{equation*}
    \Prim C^*(\G) \subset \bigcup_{i\in I} \bar{\bigcup_{\pi \in \mathcal{F}_i} \supp(\Ind_{U_i} \pi)} \subset \bar{\bigcup_{i\in I} \bigcup_{\pi \in \mathcal{F}_i} \supp(\Ind_{U_i} \pi)} = \bar{\bigcup_{\pi \in \mathcal{F}} \supp \pi}.
  \end{equation*}
  The converse inclusion is trivial. This shows that $\mathcal{F}$ is a faithful family.
\end{proof}

As a direct application, recall that a groupoid $\G$ is called \emph{metrically amenable} if the canonical morphism $C^*(\G) \to C^*_r(\G)$ is an isomorphism \cite{Ren1980}.
\begin{cor}
  Under the assumptions of Theorem \ref{thm:decomposition_of_the_spectrum}, assume that each groupoid $\G|_{U_i}$ is metrically amenable, for all $i \in I$. Then $\G$ is metrically amenable.
\end{cor}
\begin{proof}
  For each $i \in I$, let $\mathcal{F}_i = (\pi^{U_i}_x)_{x\in U_i}$ be the family of all regular representations of $\G|_{U_i}$. The groupoid $\G|_{U_i}$ is metrically amenable if, and only if, the family $\mathcal{F}_i$ is faithful for $C^*(\G|_{U_i})$. Now recall from Proposition \ref{prop:induction_main_properties} that $\Ind_{U_i} \pi_x^{U_i} = \pi_x$, which is the regular representation of $\G$ at $x$. Corollary \ref{cor:faithful_induction} implies that the family $\mathcal{F} =  (\pi_x)_{x\in X}$ is faithful for $C^*(\G)$. This in turn is equivalent to $\G$ being metrically amenable.
\end{proof}

\begin{defn}[ \citeauthor{NP2017} \cite{NP2017} ] \label{defn:exhaustive_families}
  Let $A$ be a $C^*$-algebra. A family $\mathcal{F}$ of representations of $A$ is called \emph{exhaustive} if
  \begin{equation*}
    \Prim A = \bigcup_{\pi \in \mathcal{F}} \supp \pi.
  \end{equation*}
\end{defn}
Exhaustive families provide a refinment of faithful families and will be used in Section \ref{sec:application_to_fredholm_groupoids}.

\begin{cor} \label{cor:exhaustive_families}
  We follow the assumptions of Theorem \ref{thm:decomposition_of_the_spectrum}. For each $i \in I$, let $\mathcal{F}_i$ be a exhaustive family of representations of $C^*(\G|_{U_i})$. Then the family
  \begin{equation*}
    \{ \Ind_{U_i} \pi \mid i \in I, \pi \in \mathcal{F}_i \}
  \end{equation*}
  is exhaustive for $C^*(\G)$.
\end{cor}

The proof is the same as that of Corollary \ref{cor:faithful_induction}.

\section{Fredholm groupoids}
\label{sec:application_to_fredholm_groupoids}

The class of Fredholm Lie groupoid was introduced by Carvalho, Nistor and Qiao in \cite{CNQ2017} as an important tool to study differential equations on manifolds with singularities. Our main result (Theorem \ref{thm:Fredholm_is_local}) is that the Fredholm property is \emph{local}: a groupoid is Fredholm if, and only if, all its reductions to open subsets of the unit space are also Fredholm groupoids. This motivates the definition of \emph{local action groupoids} in Subsection \ref{sub:Local_action_groupoids}, which occur naturally in many practical cases.

\subsection{Definitions}
\label{sub:fredholm_groupoids}
Let $\G \dr X$ be a locally compact, second countable, locally Hausdorff groupoid with a continuous right-invariant Haar system. Throughout this subsection, we will assume that there is a $\G$-invariant, open and dense orbit $V \subset X$ such that $\G_V \simeq V\times V$. Such a set $V$ is necessarily unique. Define the \emph{vector representation}
\begin{equation*}
  \pi_0 : C^*_r(\G) \to \B(L^2(V)),
\end{equation*}
as the equivalence class of any regular representation $\pi_x$, for any $x \in V$ (all those representations are conjugated through the action of $\G$ on its fibers $\G_x = d^{-1}(x)$).

Fredholm groupoid are tailored to study differential operators on $V$, so one usually asks $V$ to have a smooth structure: this is the case, for example, when $\G$ is a Lie groupoid, or more generally a continuous family groupoid \cite{LMN2000,Pat2001}. However, the differential setting is not needed for the results we seek; thus our definition of a Fredholm groupoids is a strict extension of the original one from \cite{CNQ2017}.

\begin{defn} \label{defn:fredholm_groupoid}
  A \emph{Fredholm groupoid} is a locally compact, second-countable, locally Hausdorff groupoid $\G \dr X$, endowed with a continuous right-invariant Haar system, such that 
  \begin{enumerate}
    \item \label{it:def1_dense_orbit} there is an open, dense $\G$-orbit $V$ such that $\G_V \simeq V\times V$,
    \item \label{it:pi_0_inj} the vector representation $\pi_0 : C^*_r(\G) \to \B(L^2(V))$ is injective, and
    \item \label{it:def1_bdry} for any $a \in C^*_r(\G)$, the operator $1 + \pi_0(a)$ is Fredholm in $\B(L^2(V))$ if, and only if, each operator $1 + \pi_x(a)$ is invertible, for every $x \in X \setminus V$.
  \end{enumerate}
\end{defn}

An equivalent definition of Fredholm groupoids was given in \cite{CNQ2017}. Recall the concept of an exhaustive family of representations from Definition \ref{defn:exhaustive_families}. 

\begin{prop} \label{prop:Fredholm_second_def}
  Let $\G \dr X$ be a locally compact, second-countable and locally Hausdorff groupoid, endowed with a continuous right-invariant Haar system. Then $\G$ is a Fredholm groupoid if, and only if, all the following conditions are met:
  \begin{enumerate}
    \item there is an open, dense $\G$-orbit $V$ such that $\G_V \simeq V\times V$,
    \item \label{it:def2_pi_0_inj} the vector representation $\pi_0 : C^*_r(\G) \to \B(L^2(V))$ is injective,
    \item \label{it:iso_quotient} the restriction map $C^*_r(\G) \to C^*_r(\G_F)$ induces an isomorphism
      \begin{equation*}
        C^*_r(\G)/C^*_r(\G_V) \simeq C^*_r(\G_F),
      \end{equation*}
      where $F = X \setminus V$, and
    \item \label{it:exhaustive_bdry} the family of representations $(\pi_x)_{x \in F}$ is exhaustive for $C^*_r(\G_F)$.
  \end{enumerate}
\end{prop}
This Proposition was proven in \cite{CNQ2017} for Lie groupoids, but without making any use of the smooth structure: thus it extends without any modification to our setting. Note that Conditions \eqref{it:iso_quotient} and \eqref{it:exhaustive_bdry} may be checked at once by stating that the family $(\pi_x)_{x\in F}$ is exhaustive for the quotient algebra $C^*_r(\G)/C^*_r(\G_U)$.

Recall the definition of a \emph{metrically amenable} groupoid from Subsection \ref{sub:families_of_representations}.

\begin{thm} \label{thm:amenable_implies_Fredholm}
  Let $\G \dr X$ be a locally compact and second-countable groupoid endowed with a continuous right-invariant Haar system. Assume that there is an open, dense and $\G$-invariant subset $V \subset X$ such that $\G_V \simeq V\times V$, and put $F=X\setminus V$. Assume moreover that $\G$ is Hausdorff and $\G_F$ metrically amenable. Then $\G$ is Fredholm.
\end{thm}
Theorem \ref{thm:amenable_implies_Fredholm} gives a sufficient condition for Fredholmness which is satisfied by many groupoids encountered in practical cases (see Subsection \ref{sub:Examples} for examples). When $\G$ is moreover a Lie groupoid, its dense orbit $V$ is called a manifold with \emph{amenable ends} \cite{CNQ2017}.

\begin{proof}
  This result was proven in \cite{CNQ2017} for Lie groupoids, so we will only give a sketch of the proof here. First, it follows from a lemma of Khoskham and Skandalis \cite{KS2002} (and the density of $V$ in $X$) that the vector representation is always injective when $\G$ is Hausdorff. This proves Condition \eqref{it:def2_pi_0_inj} of Proposition \ref{prop:Fredholm_second_def}. 
  
  The amenability of $\G_F$ and $\G_V \simeq V\times V$ imply that $\G$ is also metrically amenable. It is then a standard fact that the restriction map induces an isomorphism $C^*_r(\G)/C^*_r(\G_U) \simeq C^*_r(\G_F)$ \cite{Ren1980}, which proves Condition \eqref{it:iso_quotient}. Condition \eqref{it:exhaustive_bdry} is a result of Nistor and Prudhon: if $\G_F$ is metrically amenable, then its set of regular representations $(\pi_x)_{x\in F}$ is exhaustive for $C^*_r(\G_F)$ \cite[Theorem 3.18]{NP2017}. This follows from the Effros-Hahn conjecture, which was proven for amenable groupoids \cite{IW2009}.
\end{proof}

Many examples of Fredholm groupoids (as well as their relation with the study of differential equations on open manifolds) will be given in Section \ref{sec:Applications to differential equations}.

\subsection{The Fredholm property is local}
\label{sub:The Fredholm property is local}
Our aim in this section is to use the results of Section \ref{sec:representations_induced_from_reductions} to prove our main result, Theorem \ref{thm:main}. In a nutshell, we show that a groupoid $\G$ is Fredholm if, and only if, all its reductions to open subsets of the unit are Fredholm.

\begin{lm} \label{lm:Fredholm_goes_to_reductions}
  Let $\G \dr X$ be a Fredholm groupoid. Then, for any open set $U \subset X$, the reduction $\G|_U$ is also a Fredholm groupoid.
\end{lm}
\begin{proof}
  Let $V \subset X$ be the unique open dense $\G$-orbit such that $\G_V \simeq V\times V$, and put $F = X \setminus V$. Then $V' \eqdef U \cap V$ is the unique open dense $\G|_U$-orbit such that $\G|_{V'} \simeq V' \times V'$.

  Let $a \in C^*_r(\G|_U)$. Because $\pi_0$ is injective on $C^*_r(\G)$ and $\pi_0(C^*_r(V'\times V')) \simeq \K(L^2(V'))$, there is an induced isomorphism
  \begin{equation*}
    \pi_0 : C^*_r(\G|_{U})/C^*_r(\G|_{V'}) \simeq \pi_0(C^*_r(\G|_{U}))/\K(L^2(V')).
  \end{equation*}
  
  Therefore, for any $a \in C^*_r(\G|_U)$, the operator $1 + \pi_0(a)$ is Fredholm in $\B(L^2(V'))$ if, and only if, the class of $1+a$ is invertible in the unitarization of $C^*_r(\G|_U)/C^*_r(\G|_{V'})$. But $C^*_r(\G|_U)$ is a subalgebra of $C^*_r(\G)$, and $C^*_r(\G|_{V'}) = C^*_r(\G_V) \cap C^*_r(\G|_U)$. Thus
  \begin{equation*}
    C^*_r(\G|_U) / C^*_r(\G|_{V'}) \subset C^*_r(\G)/C^*(\G_V).
  \end{equation*}
  Hence, $1+a$ is invertible in the unitarization of $C^*_r(\G|_U) / C^*_r(\G|_{V'})$ if, and only if, it is invertible as an element of the unitarization of $C^*_r(\G)/C^*(\G_V)$. 
  
  Now, since $\G$ is a Fredholm groupoid, we deduce that $1 + \pi_0(a)$ is Fredholm in $\B(L^2(V'))$ if, and only if, the operator $1 + \pi_x(a)$ is invertible for each $x \in F$. But $\pi_x(a) =0$ for all $x \notin U$. Therefore, the operator $1 + \pi_0(a)$ is Fredholm if, and only if, the operator $1 + \pi_x(a)$ is invertible for each $x \in F \cap U = U \setminus V'$. This proves that $\G|_U$ is a Fredholm groupoid.
\end{proof}

We now establish the converse of Lemma \ref{lm:Fredholm_goes_to_reductions}.

\begin{thm} \label{thm:Fredholm_is_local}
  Let $\G\dr X$ be a locally compact, second-countable and locally Hausdorff groupoid, endowed with a right-invariant Haar system. Assume that
  \begin{enumerate}
    \item there is an open dense $\G$-invariant subset $V \subset X$ with $\G_V \simeq V\times V$, and
    \item we have a family $(U_i)_{i\in I}$ of open subsets of $X$ such that the saturations $(\G\cdot U_i)_{i\in I}$ provide an open cover of $X$.
  \end{enumerate}
  Then $\G$ is a Fredholm groupoid if, and only if, each reduction $\G|_{U_i}$ is also a Fredholm groupoid, for every $i \in I$.
\end{thm}

Theorem \ref{thm:Fredholm_is_local} is the main result of this paper. It emphasizes the fact that the Fredholmness of a groupoid $\G$ is determined by its local structure. In particular, what really matters is the local structure in a neighborhood of the closed set $F = X \setminus V$, or in other words how the groupoid $\G_F$ is glued to the pair groupoid $\G_V = V \times V$.

\begin{proof}[Proof of Theorem \ref{thm:Fredholm_is_local}]
  Assume that each reduction $\G|_{U_i}$ is a Fredholm groupoid, and let $V_i\subset U_i$ be the unique open dense $\G|_{U_i}$-orbit such that $\G|_{V_i} \simeq V_i \times V_i$. We only have to prove that $\G$ satisfies the assumptions \eqref{it:def2_pi_0_inj}, \eqref{it:iso_quotient} and \eqref{it:exhaustive_bdry} of Proposition \ref{prop:Fredholm_second_def}.

  First, for any $i \in I$, let $\pi_0^i : C^*_r(\G|_{V_i}) \to \B(L^2(V_i))$ be the vector representation of $\G|_{V_i}$. We know from Proposition \ref{prop:induction_main_properties} that $\Ind_{U_i} \pi_0^i$ is the vector representation $\pi_0$ of $C^*_r(\G)$ on $\B(L^2(V))$. Moreover, because $\G|_{U_i}$ is Fredholm, the representation $\pi_0^i$ is faithful. Corollary \ref{cor:faithful_induction} implies that $\pi_0$ is a faithful representation of $C^*(\G)$, which proves Assumption \eqref{it:def2_pi_0_inj}.

  Now, because $(\G\cdot U_i)_{i \in I}$ is an open cover of $X$, Theorem \ref{thm:decomposition_of_the_spectrum} implies that
  \begin{equation*}
    \Prim C^*_r(\G) = \bigcup_{i\in I} \Prim C^*_r(\G|_{U_i}).
  \end{equation*}
  Since $V_i$ is a $\G|_{U_i}$-invariant open subset of $U_i$, we may expand this decomposition:
  \begin{equation} \label{eq:proof_decompo_1}
    \begin{aligned}
      \Prim C^*_r(\G) & = \bigcup_{i\in I} \big( \Prim C^*_r(\G|_{V_i}) \bigsqcup \Prim \left(C^*_r(\G|_{F_i})\right)\big)\\
                      & = \left(\bigcup_{i\in I} \Prim C^*_r(\G|_{V_i})\right) \bigcup \left( \bigcup_{i \in I} \Prim \left(C^*_r(\G|_{F_i})\right) \right),
    \end{aligned}
  \end{equation}
  where we have put $F_i \eqdef U_i \setminus V_i$ and used the isomorphism $C^*_r(\G|_{U_i})/C^*_r(\G|_{V_i}) \simeq C^*_r(\G_{F_i})$ given by the fact that $\G|_{U_i}$ is a Fredholm groupoid. But the family $(\G\cdot V_i)_{i\in I}$ is an open cover of $V$, so another application of Theorem \ref{thm:decomposition_of_the_spectrum} yields
  \begin{equation*}
    \Prim C^*_r(\G_V) = \bigcup_{i \in I} \Prim C^*_r(\G|_{V_i}).
  \end{equation*}
  By substituting this last expression in Equation \eqref{eq:proof_decompo_1}, we obtain
  \begin{equation} \label{eq:proof_decompo_C_r_G}
\Prim C^*_r(\G) = \Prim C^*_r(\G_V) \bigcup \left(\bigcup_{i\in I} \Prim C^*_r(\G|_{F_i})\right),
  \end{equation}

   On the other hand, because $V$ is a $\G$-invariant open subset, there is also a decomposition
   \begin{equation} \label{eq:proof_V_invariant}
     \Prim C^*_r(\G) = \Prim C^*_r(\G_V) \bigsqcup \Prim( C^*_r(\G)/C^*_r(\G_V) )
   \end{equation}
   Combining Equations \eqref{eq:proof_decompo_C_r_G} and \eqref{eq:proof_V_invariant} proves the inclusion
  \begin{equation*}
    \Prim( C^*_r(\G)/C^*_r(\G_V) ) \subset \bigcup_{i \in I} \Prim C^*_r(\G|_{F_i}).
  \end{equation*}

  For $i \in I$ and $x \in U_i$, let us denote by $\pi_x^{i}$ the regular representation of $\G|_{U_i}$ at $x$. Recall from Proposition \ref{prop:induction_main_properties} that $\Ind_{U_i} (\pi_x^i) = \pi_x$ (with $\pi_x$ the regular representation of $\G$ at $x$), so $\Ind_{U_i}(\supp \pi_x^i) \subset \supp \pi_x$. Since $\G|_{U_i}$ is a Fredholm groupoid, the family $(\pi_x^{i})_{x \in F_i}$ is exhaustive for $C^*_r(\G|_{F_i})$; in other words $\Prim C^*_r(\G|_{F_i})$ is the union of the support of every $\supp(\pi^i_x)$, for $x \in F_i$. Therefore
  \begin{equation*}
    \Prim( C^*_r(\G)/C^*_r(\G_V) ) \subset \bigcup_{i \in I} \bigcup_{x \in F_i} \Ind_{U_i}(\supp \pi_x^i) \subset \bigcup_{x \in F} \supp \pi_x,
  \end{equation*}
  with $F \eqdef X \setminus V = \bigcup_{i \in I} F_i$. On the other hand, the representation $\pi_x$ vanishes on $C^*_r(\G_V)$ for any $x \in F$, so that $\supp \pi_x$ is contained in $\Prim( C^*_r(\G)/C^*_r(\G_V) )$. This proves the equality
    \begin{equation*}
      \Prim( C^*_r(\G)/C^*_r(\G_V) ) = \bigcup_{x \in F} \supp \pi_x,
    \end{equation*}
    which by definition indicates that the family $(\pi_x)_{x \in F}$ is exhaustive for the quotient algebra $C^*_r(\G)/C^*_r(\G_V)$. This proves Assumptions \eqref{it:iso_quotient} and \eqref{it:exhaustive_bdry} of Proposition \ref{prop:Fredholm_second_def} and concludes the proof that $\G$ is a Fredholm groupoid. Finally, the ``only if'' part of Theorem \ref{thm:Fredholm_is_local} is a consequence of Lemma \ref{lm:Fredholm_goes_to_reductions} above.
\end{proof}

\subsection{Consequences}
\label{sub:Fredholm_applications}

We give here several corollaries of Theorema \ref{thm:Fredholm_is_local}, which may be used as tools to check the Fredholm property for a given groupoid. Some concrete examples of groupoids and applications of these results will be shown in Section \ref{sec:Applications to differential equations}.

\subsubsection{Gluing groupoids}
\label{sub:Gluing groupoids}

Let $(U_i)_{i \in I}$ be an open cover of a locally compact, Hausdorff space $X$. Assume that we are being given a family of locally compact groupoids $(\G_i \dr U_i)_{i\in I}$ with isomorphisms
\begin{equation*}
  \phi_{ji} : \G_i|_{U_i \cap U_j} \to \G_j|_{U_i\cap U_j},
\end{equation*}
for all $i,j \in I$. A natural construction would be to glue this family into a \enquote{bigger} groupoid $\G \dr X$. This requires some compatibility assumptions on the family $(\G_i)_{i\in I}$, as was studied in \cite{GL2014, Com2018}. For instance, the family is said to satisfy the \emph{weak gluing condition} of \cite{Com2018} if
\begin{enumerate}
  \item the isomorphisms $(\phi_{ij})$ satisfy a cocycle condition with $\phi_{kj}\phi_{ji} = \phi_{ki}$ and $\phi_{ji} = (\phi_{ij})^{-1}$; and
  \item for any any $i,j \in I$ and any composable pair $(g_i,h_j) \in \G_i \times \G_j$ (that is, such that the domain of $g_i$ and the range of $g_j$ coincide), there is a $k \in I$ and a composable pair $(g_k,h_k) \in \G_k^{(2)}$ with $\phi_{ik}(g_k) = g_i$ and $\phi_{jk}(h_k) = h_j$.
\end{enumerate}
Under those assumptions, the fibered coproduct of the family $(\G_i)_{i \in I}$ along the isomorphisms $(\phi_{ij})_{i,j\in I}$, which we write $\G \eqdef \bigcup_{i \in I} \G_i$ acquires a natural groupoid structure over $X$ (see \cite{Com2018}). The construction of this glued groupoid is such that each reduction $\G|_{U_i}$ is naturally isomorphic to $\G_i$.

\begin{thm}
  Let $(U_i)_{i \in I}$ be an open cover of a locally compact Hausdorff space $X$, and let $(\G_i\dr U_i)_{i\in I}$ be a family of groupoids satisfying the weak gluing condition. Let $\G = \bigcup_{i \in I} \G_i$ be the glued groupoid, and assume that
  \begin{enumerate}
    \item there is an open dense $\G$-invariant subset $V \subset X$ with $\G_V \simeq V\times V$, and
    \item each groupoid $\G_i$ is Fredholm, for $i \in I$.
  \end{enumerate}
  Then the groupoid $\G$ is Fredholm.
\end{thm}
\begin{proof}
  By construction, each reduction $\G|_{U_i}$ is isomorphic to $\G_i$, hence Fredholm. Since the family $(U_i)_{i \in I}$ is an open cover of $X$, the result is a direct consequence of Theorem \ref{thm:Fredholm_is_local}.
\end{proof}

\subsubsection{Local isomorphisms}
\label{ssub:Local isomorphisms}
Theorems \ref{thm:decomposition_of_the_spectrum} and \ref{thm:Fredholm_is_local} state that the primitive spectrum of a groupoid's $C^*$-algebra can be studied locally. This suggests the following notion of \emph{local isomorphisms} between groupoids.

\begin{defn}
  Let $\G \dr X$ and $\H \dr Y$ be two topological groupoid, and let $p \in X$. 
  \begin{enumerate}
    \item A \emph{local isomorphism} between $\G$ and $\H$ around $p$ is a triplet $(U,\phi,V)$, where $U \subset X$ and $V \subset Y$ are open subsets, with $p \in U$ and 
  \begin{equation*}
    \phi : \G|_U \to \H|_V
  \end{equation*}
  is an isomorphism of topological groupoids.

\item We say that $\G$ is \emph{locally isomorphic} to $\H$ around $p$, and we write $\G \sim_p \H$, if there exists an isomorphism between $\G$ and $\H$ around $p$.
  \end{enumerate}
\end{defn}

Recall that the \emph{direct product} of two groupoids $\G\dr X$ and $\H \dr Y$ is the groupoid $\G\times\H$, with units $X \times Y$, whose structural morphisms are the direct products of those of $\G$ and $\H$.
\begin{lm} \label{lm:local_iso_product}
  Let $\G_1 \dr X_1$ and $\G_2\dr X_2$ be two topological groupoids. Let $p_1 \in X_1$ and $p_2 \in X_2$. Assume that there are groupoids $\H_1, \H_2$ such that $\G_1 \sim_{p_1} \H_1$ and $\G_2 \sim_{p_2} \H_2$. Then $\G_1 \times \G_2$ is locally isomorphic to $\H_1\times\H_2$ near $(p_1,p_2)$.
\end{lm}
\begin{proof}
  By assumptions, there are isomorphisms $\phi_1 : \G_1|_{U_1} \to \H_1|_{V_1}$ and $\phi_2 : \G_2|_{U_2} \to \H_2|_{V_2}$, with $p_1 \in U_1$ and $p_2 \in U_2$. Then $(p_1,p_2) \in U_1 \times U_2$ and
  \begin{equation*}
    \phi_1 \times \phi_2 : (\G_1\times\G_2)|_{U_1 \times U_2} \to (\H_1 \times \H_2)|_{V_1 \times V_2}
  \end{equation*}
  is an isomorphism, which proves the lemma.
\end{proof}

Subsection \ref{sub:Gluing groupoids} introduced the gluing construction of a family of groupoids. We show that gluing groupoids preserves their local structure.
\begin{lm} \label{lm:local_iso_gluing}
  Let $(U_i)_{i\in I}$ be an open cover of a topological space $X$, and let $(\G_i\dr U_i)_{i \in I}$ be a family of topological groupoids satisfying the weak gluing condition. Let $i \in I$ and $p \in U_i$, and assume that there is a groupoid $\H$ such that $\G_i \sim_p \H$. Then
  \begin{equation*}
    \bigcup_{i\in I} \G_i \sim_p \H
  \end{equation*}
\end{lm}
\begin{proof}
  Let $\G = \bigcup_{i\in I} \G_i$ be the glued groupoid. By definition, we have $\G|_{U_i} \simeq \G_i$; hence any local isomorphism $\phi : \G_i|_U \to \H|_V$ around $p$ induces a local isomorphism $\G|_{U_i\cap U} \simeq \H|_V$ around $p$.
\end{proof}

\begin{thm} \label{thm:local_iso_Fredholm}
  Let $\G\dr X$ be a locally compact, second-countable and locally Hausdorff groupoid. Assume that
 \begin{enumerate}
   \item there is an open dense $\G$-invariant subset $V \subset X$ with $\G_V \simeq V\times V$, and
   \item for each $p \in X$, there is a Fredholm groupoid $\H_p$ such that $\G \sim_p \H_p$. 
 \end{enumerate} 
  Then $\G$ is a Fredholm groupoid.
\end{thm}
The point of Theorem \ref{thm:local_iso_Fredholm} is to emphasize again that only the local structure is important to characterize Fredholm groupoids.
\begin{proof}
  Following the assumptions, there is for each $p \in X$ an open set $U_p$ containing $p$ and such that $\G|_{U_p}$ is isomorphic to a reduction $\H_p|_{V_p}$, with $\H_p$ a Fredholm groupoid. Lemma \ref{lm:Fredholm_goes_to_reductions} implies that $\H_p|_{V_p}$ is Fredholm, so $\G|_{U_p}$ is also Fredholm. The conclusion follows from Theorem \ref{thm:Fredholm_is_local} applied to the open cover $(U_p)_{p \in X}$.
\end{proof}

\subsubsection{Local action groupoids}
\label{sub:Local_action_groupoids}
Many Fredholm groupoids occuring in the study of differential equation on singular spaces are very simple on a local scale: they are locally isomorphic to action groupoids. To formalize this, we introduce here the class of \emph{local action groupoids}. Many examples of such groupoids will be found in Subsection \ref{sub:Examples} below.

\begin{rem}
Recall that, if $G$ is a group acting on a space $X$ on the right, then the corresponding \emph{action groupoid} (or transformation groupoid) is written $X \rtimes G$ and defined as follows. As a set, we put $X\rtimes G \eqdef X \times G$. The domain and range maps are given by
\begin{equation*}
  d(x,g) = x\cdot g^{-1} \quad \text{and} \quad r(x,g) = x,
\end{equation*}
whereas the product is $(x,g)(x\cdot g^{-1},h) = (x,hg)$ (see \cite{Ren1980,Com2018} for more details). 

If $G$ and $X$ are both locally compact, second-countable and Hausdorff, and if moreover the action is continuous, then $X \rtimes G$ is a locally compact, second-countable, Hausdorff groupoid. The groupoid $X \rtimes G$ is endowed with a natural Haar system (induced by the Haar measure on $G$), and the reduced groupoid $C^*$-algebra $C^*_r(X \rtimes G)$ is then isomorphic to the crossed-product algebra $C_0(X)\rtimes_r G$. 
\end{rem}

\begin{thm} \label{thm:action_groupoid_Fredholm}
  Let $G$ be a topological group acting continuously on a space $X$. Assume that $G$ and $X$ are locally compact, Hausdorff and second-countable. Assume moreover that:
  \begin{enumerate}
    \item there is an open, dense $G$-orbit $V \subset X$ such that the action of $G$ on $V$ is free, transitive and proper, 
    \item the group $G$ is amenable.
  \end{enumerate}
  Then the groupoid $X \rtimes G$ is Fredholm.
\end{thm}

Action groupoids of the sort occur naturally in the study of the essential spectrum of differential operators on groups, a notable example being that of the quantum $N$-body problem on Euclidean space \cite{GI2006,MNP2017,MPR2007}. 

\begin{proof}
  Let $\G = X \rtimes G$. Firstly, the assumptions on the action of $G$ on $V$ imply that the map
  \begin{align*}
    \alpha: V \rtimes G & \to V \times V \\
    (x,g) & \mapsto (x,x\cdot g^{-1})
  \end{align*}
  is continuous and bijective. Moreover $\alpha$ is proper with value in a Hausdorff space, hence closed. Therefore $\alpha$ is an homeomorphism, which shows that $\G_V \simeq V\times V$ Secondly, the amenability of $G$ implies that the groupoid $\G_F = F \rtimes G$ is metrically amenable \cite{Wil2007}, where $F = X \setminus G$. The result follows from Theorem \ref{thm:amenable_implies_Fredholm}.
\end{proof}

\begin{defn} \label{defn:local_action_groupoids}
  A locally compact and second countable groupoid $\G \dr X$ is said to be a \emph{local action groupoid} if, for each $p \in X$, there is an action groupoid $X_p \rtimes G_p$ such that $\G$ is locally isomorphic to $X_p \rtimes G_p$ near $p$, where $G_p$ and $X_p$ are both locally compact, Hausdorff and second countable.
\end{defn}

The main point of Definition \ref{defn:local_action_groupoids} is that the local structure of a such groupoids is very well understood: indeed, the $C^*$-algebras of action groupoids and their representations have been much studied in the litterature \cite{Wil2007}. Several examples of local action groupoids shall be given in Section \ref{sec:Applications to differential equations}. 

\begin{prop} \label{prop:local_action_product}
  Let $\G$ and $\H$ be local action groupoids. Then $\G \times \H$ is also a local action groupoid.
\end{prop}
\begin{proof}
  This follows from Lemma \ref{lm:local_iso_product} and the fact that, if $\G_1 = X_1\rtimes G_1$ and $\G_2 = X_2\rtimes G_2$ are action groupoids, then 
  \begin{equation*}
    \G_1 \times \G_2 \simeq (X_1 \times X_2) \rtimes (G_1 \times G_2),
  \end{equation*}
  where $G_1\times G_2$ acts on $X_1 \times X_2$ by the product action.
\end{proof}

\begin{prop} \label{prop:local_action_gluing}
  Let $(\G_i)_{i\in I}$ be a family of local action groupoids satisfying the weak gluing condition of Subsection \ref{sub:Gluing groupoids}. Then the glued groupoid $\G = \bigcup_{i \in I} \G_i$ is also a local action groupoid.
\end{prop}
\begin{proof}
  This is a direct consequence of Lemma \ref{lm:local_iso_gluing}.
\end{proof}

Since \enquote{the Fredholm property is local}, it is only natural that Theorem \ref{thm:action_groupoid_Fredholm} generalizes to local action groupoids.

\begin{thm} \label{thm:local_action_groupoids_are_Fredholm}
  Let $\G \dr X$ be a local action groupoid. Assume that
  \begin{enumerate}
    \item there is an open, dense and $\G$-invariant subset $V \subset X$ such that $\G_V \simeq V \times V$; 
    \item for each $p \in X$, there is a local isomorphism $\G \sim_p X_p \rtimes G_p$, with $X_p$, $G_p$ locally compact, second countable and Hausdorff, and $G_p$ \emph{amenable}.
  \end{enumerate}
  Then $\G$ is a Fredholm groupoid.
\end{thm}
In other words, if $\G$ is locally given by the action of an \emph{amenable} group, then $\G$ is Fredholm. Section \ref{sec:Applications to differential equations} will provide many examples of such groupoids.

\begin{proof}
  Set $p \in X$, and consider the local isomorphism $\phi_p : \G|_{U_p} \simeq (X_p \rtimes G_p)|_{U_p'}$, where $p \in U_p$. Let $V_p = V \cap U_p$ and $V'_p = \phi_p(V_p)$. Because $\G_V \simeq V\times V$, we have $\G|_{V_p} \simeq V_p \times V_p$, hence
  \begin{equation*}
    (X_p\rtimes G_p)|_{V'_p} \simeq V'_p \times V'_p.
  \end{equation*}

  Now, because $G_i$ is amenable, the groupoid $X_p \rtimes G_p$ is Hausdorff and amenable. It follows that its reduction $(X_p \rtimes G_p)|_{U'_p}$ is also Hausdorff and amenable. Theorem \ref{thm:amenable_implies_Fredholm} therefore implies that each groupoid $(X_p \rtimes G_p)|_{U_p}$ is Fredholm. We conclude using Theorem \ref{thm:local_iso_Fredholm} that $\G$ is a Fredholm groupoid.
\end{proof}

\section{Examples and Applications}
\label{sec:Applications to differential equations}

We conclude this paper with many examples of Fredholm and local action groupoids. All our examples are motivated by the study of differential equations on singular spaces, so we begin in Subsection \ref{sub:Differential operators} by discussing the notion of differential operators induced by a Lie groupoid.

\subsection{Differential operators}
\label{sub:Differential operators}

Many Fredholm groupoids that are studied in practical cases are Lie groupoids, i.e.\ groupoids with a smooth structure. We work in the setting of manifolds with corners, in other words manifolds which are modelled on open subsets of the cube $[0;1]^n$, for $n \in \N$ (see \cite{Joy2012} for more on that matter). A Lie groupoid in that context is defined as follows.

\begin{defn} \label{defn:lie_groupoids}
  A groupoid $\G\dr X$ is a \emph{Lie groupoid} if
  \begin{enumerate}
    \item both $\G$ and $X$ are manifolds with corners, with $X$ Hausdorff,
    \item the structural maps $d,r,\iota$ and $u$ are smooth,
    \item the map $d$ is a \emph{tame} submersion, and
    \item the composition map $\mu$ is smooth.
  \end{enumerate}
\end{defn}
A submersion $h : X\to Y$ of manifolds with corners is said to be \emph{tame} if, for all $v \in TX$, the vector $h_*(v) \in TY$ is inward-poiting if and only $v$ is. If $\G$ is a Lie groupoid with units $X$, the tameness condition ensures that the fibers $\G_x = d^{-1}(x)$, for $x \in X$, are smooth manifolds without corners \cite{Joy2012}. Note that Definition \ref{defn:lie_groupoids} does not require $\G$ to be Hausdorff; however, because $\G$ is a manifold, it is always locally Hausdorff.

The \emph{Lie algebroid} of a Lie groupoid $\G\dr X$ is the bundle $A\G \to X$ of all vectors tangent to the $d$-fibers $\G_x$, for $x \in X$; in other words 
\begin{equation*}
  A\G = \bigcup_{x \in X} T_x\G_x = (\ker d_*)|_X.
\end{equation*}
The differential of the range $r : \G \to X$ gives the \emph{anchor map} $r_* : A\G \to TX$. \cite{Mac1987}

Assume now that $X$ is \emph{compact} and that there is an open, dense and $\G$-invariant subset $V \subset X$ such that $\G_V \simeq V\times V$. Then $r_*$ is an isomorphism from $A\G|_V$ to $TV$. Thus any metric $g$ on $A\G$ induces a Riemannian metric $g_0$ on $V$, which we call \emph{compatible} with $A\G$. Such metrics $g_0$ are always complete, and their equivalence class depends on $A\G$ only \cite{ALN2004}. Associated to $g_0$ is a well-defined scale of \emph{Sobolev spaces} $(H^s(V))_{s \in \R}$, which all contain $C_c^\infty(V)$ as a dense subset \cite{HR2008}. Naturally $H^0(V) = L^2(V)$.

\begin{defn} \label{defn:differential_operators}
  A \emph{differential operator} $P$ on a Lie groupoid $\G$ with compact units $X$ is a family $(P_x)_{x\in X}$ such that
  \begin{enumerate}
    \item for any $x \in X$, the element $P_x$ is a differential operator on $\G_x = d^{-1}(x)$,
    \item (right-invariance) for any $g \in \G$, the right multiplication $R_g : \G_{r(g)} \to \G_{d(g)}$ gives a conjugation $R_{g^{-1}}^* P_{r(g)} R_g^* = P_{d(g)}$, and
    \item (smoothness) for any $f \in C^\infty(\G)$, the map $x \mapsto P_x f_x$ is smooth, where $f_x = f|_{\G_x}$.
  \end{enumerate}
\end{defn}
We let $\Diff(\G)$ be the algebra of all differential operators on $\G$, and $\Diff^m(\G)$ the subspace of operators of order lesser or equal to $m\in\N$. Operators of order $1$ are just right-invariant vector fields on $\bigcup_{x\in X}T\G_x$, which are in one-to-one correspondence with $\Gamma(A\G)$. Thus the algebra $\Diff(\G)$ may be alternatively described as the universal envelopping algebra of the Lie algebroid $A\G$ \cite{NWX1999}.

The anchor map $r_* : \Gamma(A\G) \to \Gamma(TX)$, whose image can be restricted to $U$, induces an injective algebra morphism
\begin{equation*}
  \pi_0 : \Diff(\G) \to \Diff(V),
\end{equation*}
whose image we denote $\Diff_\G(V)$ (correspondingly, we write $\Diff^m_\G(V)$ for the image of $\Diff^m(\G)$ through $\pi_0$). Operators in $\Diff^1_\G(V)$ (that is, sections of $A\G$) should really be thought of as the \enquote{infinitesimal action} of $\G$ on its dense orbit $V$. 

Two important properties of $\Diff_\G(V)$ where shown in \cite{ALN2004}. First, the algebra $\Diff_\G(V)$ contains every geometric operator associated to the compatible metric $g_0$ (such as the Laplacian and any generalized Dirac operator). Secondly, any differential operator $P \in \Diff^m_\G(V)$ induces a bounded operator $P : H^s(V) \to H^{s-m}(V)$, for any $s \in \R$.

The main motivation for introducting and studying Fredholm groupoids is to obtain Fredholm conditions for the operators in $\Diff_\G(V)$. This is illustrated by the following theorem.

\begin{thm}[\citeauthor{CNQ2017} \cite{CNQ2017}] \label{thm:fredholm_condition_for_operators}
  Let $\G \dr X$ be a Fredholm Lie groupoid with compact unit space $X$, and set $V \subset X$ its unique dense, open $\G$-orbit. Let $P \in \Diff^m_\G(V)$ and set $s \in \R$. Then $P :H^s(V) \to H^{s-m}(V)$ is Fredholm if, and only if,
  \begin{enumerate}
    \item $P$ is elliptic and
    \item for any $x \in X\setminus V$, the operator $P_x : H^s(\G_x) \to H^{s-m}(\G_x)$ is invertible.
  \end{enumerate}
\end{thm}

Recall that a differential operator $P$ on $V$ is called \emph{elliptic} if its principal symbol $\sigma(P) \in C^\infty(T^*M)$ vanishes only on the zero-section \cite{Hor1985}. The operators $P_x$, for $x \in X\setminus V$, are called \emph{limit operators} for $P$: they are invariant under the action of $\G$ on $\G_x$, and are of the same type as $P$ (e.g.\ if $P$ is the Laplacian on $V$, then $P_x$ is the Laplacian on $\G_x$). Note that Theorem \ref{thm:fredholm_condition_for_operators} remains true if we condider pseudodifferential operators on $\G$ or operators acting between vector bundles sections \cite{CNQ2017}. Many similar results were known in particular cases, see \cite{DLR2015,Sch1991,Dau1988,Kon1967,KMR1997, GI2006, MPR2007,Man2002} and the reference therein. 

\begin{rem}
  For many practical applications, it may be interesting to extend the definition of differential operators (Definition \ref{defn:differential_operators}) and the corresponding Fredholm and boundedness results to the setting of groupoids that are only longitudinally smooth: see \cite{LMN2000,Pat2001}.
\end{rem}

\subsection{Examples}
\label{sub:Examples}
We conclude this paper with many examples of groupoids occuring in practical applications, and show that these groupoids are all local action groupoids, and in particular Fredholm.

\subsubsection{The pair groupoid}
\label{ssub:The pair groupoid}
Let $M$ be a closed manifold, i.e.\ a compact smooth manifold without boundary. The pair groupoid $\G = M\times M$ is then Fredholm: indeed, the vector representation
\begin{equation*}
  \pi_0 : C^*_r(M\times M) \to \B(L^2(M))
\end{equation*}
is an isomorphism onto the ideal $\K$ of compact operators on $L^2(M)$ \cite{CNQ2017}. Thus the operators $1 + \pi_0(a)$, for $a \in C^*_r(\G)$, are all Fredholm. Assumption \eqref{it:def1_bdry} of Definition \ref{defn:fredholm_groupoid} is trivially satisfied in that case, because the boundary set $M \setminus M$ is empty.

Here the algebra $\Diff_\G(M)$ contains \emph{all} differential operators on $M$. Theorem \ref{thm:fredholm_condition_for_operators} then recovers the classical Fredolmness result: a differential operator on $M$ is Fredholm if, and only if, it is elliptic.
The groupoid $\G$ is a local action groupoid. Indeed, any point $p \in M$ has a neighborhood $U \subset M$ diffeomorphic to an open subset $U' \subset \R^n$. Then
\begin{equation*}
  \G|_U \simeq U'\times U' \simeq (\R^n \rtimes_\alpha \R^n)|_{U'},
\end{equation*}
where $\alpha$ is the action of $\R^n$ on itself by translation.

\subsubsection{Cylindrical ends}
\label{ssub:Cylindrical ends}
Consider the smooth action of $\R$ on $\R_+$ given by
\begin{align*}
  \alpha : \R\times\R_+ & \to \R_+ \\
  (t,x) & \mapsto e^t x.
\end{align*}
The fundamental vector fields for this action are the image of the sections of the Lie algebroid of $\R_+\rtimes\R$ through the anchor map; here they are generated by $x\d_x$. 

Let now $M$ be a manifold with boundary, and $M_0$ its interior. The $b$-groupoid over $M$, which will be denoted $\G_b$, can be obtained in several steps.
\begin{enumerate}
  \item First, let $\H \dr \R_+ \times \d M$ be the product groupoid
    \begin{equation*}
      \H = (\R_+\rtimes_\alpha\R)\times(\d M \times \d M).
    \end{equation*}
   The action of $\R$ on $\R_+$ is free and transitive on $\R_+^*$, so $\H|_{]0,1[\times\d M}$ is isomorphic to the pair groupoid of $]0,1[\times \d M$.
   \item Let $U \simeq [0,1[ \times \d M$ be a tubular neighborhood of $\d M$ in $M$. We now have an identification over the interior:
     \begin{equation*}
       \H|_{]0,1[\times\d M} \simeq (]0,1[\times\d M)^2 \simeq (M_0 \times M_0)|_{U\setminus\d M}.
     \end{equation*}
     So $\H$ and $M_0\times M_0$ may be glued into a groupoid $\G_b \dr M$ (refer to Subsection \ref{sub:Gluing groupoids} for more details concerning the gluing procedure).
\end{enumerate}

We call $\G_b$ the \emph{b-groupoid} of $M$. Any metric $g_0$ on $M_0$ that is compatible with $\G_b$ is bi-Lipschitz equivalent to the product metric
\begin{equation*}
  \frac{dx^2}{x^2} + h_{\d M}
\end{equation*}
on $U \simeq [0,1[\times\d M$, with $h_{\d M}$ a metric on $\d M$. Thus $\G_b$ models \emph{manifolds with cylindrical ends}. The algebra $\Diff_{\G_b}(M_0)$ is that of every differential operator $P$ on $M_0$ which can be written as
\begin{equation*}
  P = \sum_{|\alpha| \le m} a_\alpha (x\d_x)^{\alpha_1} \d_{y_2}^{\alpha_2}\ldots\d_{y_n}^{\alpha_n},
\end{equation*}

locally near any point of $\d M$, with $a_\alpha \in C^\infty(M)$ and $(\d_{y_i})_{i=2}^n$ a local basis of $\Gamma(T\d M)$. It contains in particular any geometric operator associated to the metric $g_0$. The algebra $\Diff_b(M_0)$ has been extensively studied, and is closely related to the $b$-calculus of Melrose and the Atiyah-Patodi-Singer index theorem of manifolds with boundaries \cite{Mel1993,APS1976}.

The groupoid $\G_b$ is obtained by gluing Hausdorff groupoids together, hence $\G_b$ is Hausdorff \cite{Com2018}. Moreover, the groupoid $\G_b$ is obtained by taking direct products and gluing together several local action groupoids: it follows from Propositions \ref{prop:local_action_product} and \ref{prop:local_action_gluing} that $\G_b$ is also a local action groupoid. The local structure is very simple: for any $p \in M$, we have a local isomorphism
\begin{equation*}
  \G \sim_p (\R_+\times\R^{n-1}) \rtimes (\R\times\R^{n-1})
\end{equation*}
The action is given by the product action of $\R^{n-1}$ on itself (by translation) and $\R$ on $\R_+$ (by $\alpha$). Since the acting groups is amenable, we conclude from Theorem \ref{thm:local_action_groupoids_are_Fredholm} that $\G_b$ is a Fredholm groupoid. This is by no mean a new result \cite{CNQ2017}, but should serve as a case in point to emphasizes the relevance of local action groupoids in practical cases.

\begin{rem}
  Carvalho and Qiao have recently considered a groupoid very close in nature to the $b$-groupoid and more suited to the study of boundary layer potential operators on manifolds with boundary \cite{CQ2018}. Because the construction is very similar, it is clear that their groupoid is also a local action groupoid (with the same local picture as for $\G_b$).
\end{rem}

\subsubsection{Cusp metrics}
\label{ssub:Cusp groupoid}
Example \ref{ssub:Cylindrical ends} can be generalized by replacing the action of $\R$ on $\R_+$ by a more general one. For example, let $\varphi$ be a function in $C^0(\R_+,\R_+)$, vanishing only at $0$ and such that
\begin{enumerate}
  \item $\varphi \in C^\infty(\R_+^*)$, and
  \item $\varphi'$ is bounded on $\R_+^*$.
\end{enumerate}
Let $\alpha : \R \times \R_+ \to \R_+$ be the flow associated to the continuous vector field $x \mapsto \varphi(x)\d_x$ on $\R_+$. The function $\varphi$ is globally Lipschitz, so this flow is well-defined for any time $t \in \R$. A typical example is any function $\varphi \in C^0(\R_+ \cap C^\infty(\R_+^*)$ satisfying
\begin{equation*}
  \left\{ \begin{array}{ll}
       \varphi(x) = x^r & \text{if $x \in [0,1]$, and} \\
       \varphi(x) = 1 & \text{if $x\ge 2$},
  \end{array}\right.
\end{equation*}
for any $r \in [1;+\infty)$. 

Under these assumptions, the open set $\R_+^*$ is an orbit of $\alpha$ on which the action is free and transitive, so we may construct a groupoid $\G_\varphi \dr M$ by following the same gluing procedure as in Example \ref{ssub:Cylindrical ends}.  The groupoid $\G_\varphi$ is a local action groupoid that is locally isomorphic to $(\R_+\times\R^{n-1}) \rtimes_\alpha (\R\times\R^{n-1})$, hence it is Fredholm by Theorem \ref{thm:local_action_groupoids_are_Fredholm}. 

The compatible metrics are bi-Lipschitz equivalent to the complete metric
\begin{equation*}
  g_0(x,p) = \frac{dx^2}{(\varphi(x))^2} + h_{\d M}
\end{equation*}
when $(x,p)$ is in a tubular neighborhood $[0,1[\times\d M$ of $\d M$, with $x > 0$. This models \emph{manifolds with cusps}, see e.g.\ \cite{RST2004,KMR1997,Dau1988}. The differential operators $P \in \Diff_{\G_\varphi}(M_0)$ can be written
\begin{equation*}
  P = \sum_{|\alpha|\le m} a_\alpha (\varphi(x)\d_x)^{\alpha_1}\d_{y_2}^{\alpha_2}\ldots\d_{y_n}^{\alpha_n},
\end{equation*}
locally near any point of $\d M$, with $a_\alpha \in C^\infty(M)$ and $(\d_{y_i})_{i=2}^n$ a local basis of $\Gamma(T\d M)$. However, if $\varphi$ is not smooth at $x=0$, then the groupoid $\G_\varphi$ is not a Lie groupoid, but only a continuous family groupoid \cite{Pat2000, LMN2000}. Obtaining Fredholm conditions for operators in $\Diff_{\G_\varphi}(M_0)$ therefore requires an extension of Theorem \ref{thm:fredholm_condition_for_operators} to the setting of continuous family groupoids.

\subsubsection{Scattering manifolds}
\label{ssub:Scattering groupoid}

Let $M$ be a compact manifold with boundary and interior $M_0$. A \emph{scattering metric} (with respect to $M$) is a Riemannian metric $g_0$ on $M_0$ such that, for a tubular neighborhood $U \simeq [0,1[ \times \d M$ of $\d M$ in $M$, we can write
\begin{equation*}
  g_0(x,p) = \frac{dx^2}{x^4} + \frac{h_{\d M}}{x^2},
\end{equation*}
for any $(x,p) \in ]0,1[\times \d M \simeq U\cap M_0$, and with $h_{\d M}$ a metric on $\d M$. A typical example is given by the stereographic compactification of $\R^n$ into the half-sphere $\S^n_+$: the euclidean metric on $\R^n$ is then a scattering metric for $\S^n_+$ \cite{Mel1995,Vas2001}. For this reason, manifolds with scattering metrics are also called \emph{asymptotically euclidean}.

The algebra of scattering differential operators, written $\Diff_{sc}(M_0)$, is the one containing all differential operators $P$ on $M_0$ that can be written
\begin{equation*}
  P = \sum_{|\alpha| \le m} a_\alpha (x^2\d_x)^{\alpha_1} (x\d_{y_2})^{\alpha_2}\ldots(x\d_{y_n})^{\alpha_n},
\end{equation*}
locally near any point of $\d M$, with $a_\alpha \in C^\infty(M)$ and $(\d_{y_i})_{i=2}^n$ a local basis of $\Gamma(T\d M)$. It contains in particular the Laplacian associated to $g_0$.

It was shown in \cite{CNQ2017,Com2018} that there is a groupoid $\G_{sc} \dr M$ such that $\Diff_\G(M_0) = \Diff_{sc}(M_0)$. Moreover, the groupoid $\G_{sc}$ can be constructed by gluing reductions of the action groupoid $\S^n_+ \rtimes \R^n$, where the smooth action of $\R^n$ on $\S^n_+$ is the only one extending the action of $\R^n$ on itself by translation. Thus $\G_{sc}$ is a local action groupoid that is locally isomorphic to $\S^n_+ \rtimes \R^n$ around any point. It follows that $\G_{sc}$ is Fredholm by Theorem \ref{thm:local_action_groupoids_are_Fredholm}. This groupoid and closely related ones were studied in \cite{Vas2001,MNP2017} for instance, in relation with the study of the spectrum of the $N$-body problem on Euclidean space.

\subsubsection{Asymptotically hyperbolic}
\label{ssub:Hyperbolic groupoid}

As in Example \ref{ssub:Scattering groupoid}, let $M$ be a compact manifold with boundary and $M_0$ its interior. An \emph{asymptotically hyperbolic metric} (with respect to $M$) is a Riemannian metric $g_0$ on $M_0$ that can be written
\begin{equation*}
  g_0(x,p) = \frac{dx^2 + h_{\d M}}{x^2} 
\end{equation*}
in a tubular neighborhood $[0,1[\times \d M$, for $x > 0$ (here $h_{\d M}$ is a metric on $\d M$, as before). A typical example is the hyperbolic space $\mathbb{H}^n$, with its usual metric, compactified into the Poincaré ball. The interesting operators are those that can be written
\begin{equation*}
  P = \sum_{|\alpha| \le m} a_\alpha (x\d_x)^{\alpha_1} (x\d_{y_2})^{\alpha_2}\ldots(x\d_{y_n})^{\alpha_n},
\end{equation*}
locally near any point of $\d M$, with $a_\alpha \in C^\infty(M)$ and $(\d_{y_i})_{i=2}^n$ a local basis of $\Gamma(T\d M)$. These operators form an algebra, which we call $\Diff_0(M_0)$. This algebra and related pseudodifferential calculi were studied in \cite{Maz1991,Sch1991,Lau2003,DLR2015} for instance.

It was shown in \cite{CNQ2017, Com2018} that there is a groupoid $\G_0 \dr M$ such that $\Diff_{\G_0}(M_0) = \Diff_0(M_0)$. Moreover, this groupoid can be obtained by gluing reductions of the action groupoid
\begin{equation*}
  X_n \rtimes G_n \eqdef (\R_+\times\R^{n-1}) \rtimes (\R_+^* \ltimes \R^{n-1}).
\end{equation*}
Here $\R_+^*$ acts on $\R^{n-1}$ by dilation, and the action of $G_n \eqdef \R_+^*\ltimes\R^{n-1}$ on itself by right multiplication extends smoothly to the boundary by the formula
\begin{equation*}
  (x_1,\ldots,x_n) \cdot (t, \xi_2,\ldots,\xi_n) = (t x_1, x_2 + x_1 \xi_2,\ldots, x_n + x_1 \xi_n), 
\end{equation*}
for all $(x_1,\ldots,x_n) \in \R^n$ and $(t,\xi_2,\ldots,\xi_n) \in G_n$. This entails that $\G_0$ is a (Hausdorff) local action groupoid, which is locally isomorphic to $X_n \rtimes G_n$ around each point of $M$. Because $G_n$ is amenable, Theorem \ref{thm:local_action_groupoids_are_Fredholm} again implies that $\G_0$ is a Fredholm groupoid.

\printbibliography
\end{document}